\documentclass{amsart}

\usepackage{amsmath}
\usepackage{amssymb}
\usepackage{amscd}
\usepackage{hyperref}

\title[The growth of Betti numbers]
 {The growth of Betti numbers \\ and approximation theorems}

\author[S. Kionke]{Steffen Kionke}

\address{Steffen Kionke\\
         Mathematisches Institut\\
         Heinrich-Heine-Universit\"at\\
         Universit\"ats\-str.~1\\
         40225 D\"usseldorf\\ Germany}
\email{steffen.kionke@uni-duesseldorf.de}
\subjclass[2010]{Primary 55N10; Secondary 20J06, 57S30}
\keywords{$L^2$-invariants, Betti numbers, approximation theorems}

\theoremstyle{plain}
\newtheorem{theorem}{Theorem}
\newtheorem{lemma}[theorem]{Lemma}
\newtheorem{corollary}[theorem]{Corollary}

\newtheorem{conjecture}[theorem]{Conjecture}

\theoremstyle{definition}
\newtheorem{definition}[theorem]{Definition}
\newtheorem{remark}[theorem]{Remark}
\newtheorem{example}[theorem]{Example}
\newtheorem{exercise}[theorem]{Exercise}
\newtheorem{question}[theorem]{Question}

\numberwithin{equation}{section}
\numberwithin{theorem}{section}
\DeclareMathOperator{\id}{Id}
\DeclareMathOperator{\im}{im}

\DeclareMathOperator{\Tr}{Tr}

\DeclareMathOperator{\Aut}{Aut}

\DeclareMathOperator{\fix}{Fix}

\DeclareMathOperator{\Mat}{Mat}

\DeclareMathOperator{\rnk}{rk}

\DeclareMathOperator{\coker}{coker}

\DeclareMathOperator{\red}{red}

\providecommand{\calH}{\mathcal{H}}
\providecommand{\calN}{\mathcal{N}}
\providecommand{\calM}{\mathcal{M}}

\providecommand{\bbN}{\mathbb{N}}
\providecommand{\bbR}{\mathbb{R}}
\providecommand{\bbQ}{\mathbb{Q}}
\providecommand{\bbZ}{\mathbb{Z}}
\providecommand{\bbF}{\mathbb{F}}

\providecommand{\bbC}{\mathbb{C}}

\providecommand{\Iwa}[1]{\bbF_p[\![#1]\!]}

\DeclareMathOperator{\GL}{GL}
\DeclareMathOperator{\RG}{RG}
\DeclareMathOperator{\U}{U}
\begin{document}

\begin{abstract}
  These notes provide a short introduction to the asymptotic behaviour of Betti number in towers of finite sheeted covering spaces.
  For rational Betti numbers a convenient answer is given by L\"uck's approximation theorem. We give a proof of L\"uck's theorem, 
  discuss generalizations and mention some related open problems. 
  Then we proceed to discuss the growth of mod-$p$ Betti numbers, where many problems remain open. We take a closer look at the special case of $p$-adic analytic towers and discuss
  a strong approximation theorem due to Bergeron-Linnell-L\"uck-Sauer and Calegari-Emerton.
\end{abstract}

\maketitle

These lecture notes provide a short introduction to the field of homology growth. They are composed out of two lectures, 
which I have given at the Borel seminar 2017 on ``\emph{Growth in Geometry and Topology}'' in Les Diablerets. These two lectures were complemented by the wonderful 
course of Holger Kammeyer on $L^2$-Invariants. Here we will mainly use the theory of $L^2$-Betti numbers; we refer to
L\"uck's book \cite{LuckBook} and to Kammeyer's notes on $L^2$-invariants \cite{HolgerL2} for an introduction. This survey does not contain 
new results (and we do not claim originality).

\bigskip

We first describe the main question. Let $X_0$ be a connected CW-complex of finite type, i.e., there are only finitely many cells of every given dimension.
Suppose we are given a \emph{tower} $(X_i)_{i\in\bbN}$ of connected covering spaces of $X_0$:
\begin{equation*}
  \dots \longrightarrow X_3 \longrightarrow X_2 \longrightarrow X_1 \longrightarrow X_0.
\end{equation*}
In the area of homology growth one is concerned with the following subject:
\begin{center}
 How do \emph{homological invariants} of the $X_i$ behave as $i$ tends to infinity?
\end{center}
In this course we will only consider the asymptotic behaviour of Betti numbers (rational, $L^2$ or mod-$p$).

To set the ground, we begin with a basic reformulation.
Let $\widetilde{G} = \pi_1(X_0)$ denote the fundamental group and let $\widetilde{X}$ be the universal covering space of $X_0$.
The fundamental group $\widetilde{G}$ acts freely on $\widetilde{X}$ and $X_0 \cong \widetilde{G} \backslash \widetilde{X}$.
Let $\widetilde{G}_i$ be the characteristic subgroup associated to the covering projection $X_i \to X$, then
\begin{equation*}
    X_i \cong \widetilde{G}_i \backslash \widetilde{X}.
\end{equation*}
For simplicity, we assume that all the covering projections $X_i \to X_0$ are \emph{normal} (sometimes called Galois coverings), so that
the subgroups $\widetilde{G}_i \trianglelefteq \widetilde{G}$ are normal.
We have $\widetilde{G}_{i+1} \subseteq \widetilde{G}_{i}$ for all $i$ and we define $N = \bigcap_{i\in\bbN} \widetilde{G}_i$.
We set $X = N\backslash \widetilde{X}$, $G = \widetilde{G}/N$ and $G_i = \widetilde{G}_i / N$. 
Thus, we can always place ourselves in the
following situation:
\begin{quote}
   There is a group $G$ and a free $G$-CW-complex $X$ such that $X_0 = G\backslash X$ is of finite type. There is a normal tower of subgroups
   $(G_i)_{i\in\bbN}$ in $G$ such that the spaces $X_i = G_i \backslash X$ form a tower of normal covering spaces over $X_0$.
\end{quote}

\begin{definition}
  Let $G$ be a group. 
  A \emph{(finite index) normal tower} in $G$ is a decreasing sequence 
  \begin{equation*}
      G \supseteq G_1 \supseteq G_2 \supseteq G_3 \supseteq \dots
  \end{equation*}
  of (finite index) normal subgroups $G_i \trianglelefteq G$ such that $\bigcap_{i\in \bbN} G_i = \{1\}$.
\end{definition}

A group $G$ is \emph{residually finite}, if for every non-trivial element $g\in G\setminus\{1\}$ there is a finite index subgroup $H \leq_{f.i.} G$ which does not contain $g$.
For example, every finitely generated linear group is residually finite (Mal'cev 1940).
If a group $G$ admits a finite index normal tower, then $G$ is clearly residually finite.
Conversely, every countable, residually finite group has a finite index normal tower.

\begin{exercise}
(a) Show that every countable, residually finite group has a finite index normal tower.

\smallskip

(b) Let $F_d$ be the free group on $d\in \bbN$ generators. Prove that $F_d$ is residually finite.
\end{exercise}

\section{L\"uck's approximation theorem}

\subsection{L\"uck's approximation theorem}

\begin{theorem}[L\"uck \cite{Luck1994}]
  Let $X$ be a free $G$-CW-complex such that $G \backslash X$ is of finite type. Every finite index normal tower $(G_i)_{i\in\bbN}$ in $G$ satisfies
  \begin{equation*}
     \lim_{i\to\infty} \frac{b_q(G_i \backslash X; \bbQ)}{[G:G_i]} = b_q^{(2)}(X;\calN(G))
  \end{equation*}
  for every $q \in \bbN_0$, where $b_q^{(2)}(X;\calN(G))$ denotes the $q$-th $L^2$-Betti number.
  In particular, the limit on the left hand side exists and is independent of the chosen tower.
\end{theorem}

\begin{remark}
 Before the work of L\"uck only \emph{Kazhdan's inequality} 
   \begin{equation*}
     \limsup_{i\to\infty} \frac{b_q(G_i \backslash X; \bbQ)}{[G:G_i]} \leq b_q^{(2)}(X;\calN(G))
   \end{equation*}
   was known.
\end{remark}

\begin{example}
  (a) Let $G = F_d$ be the free group on $d$ generators. We consider the space $X = EG$ (the $2d$-regular tree) with the action of $G$.
   We choose a finite index normal tower $(G_i)_i$ in $G$.
   By the theorem of Nielsen-Schreier, the group $G_i$ is a free group of rank $d(G_i) = (d-1)[G:G_i] + 1$.
   The space $G_i\backslash X$ is the classifying space of $G_i$ and hence $b_1(G_i,\bbQ) = d(G_i)$.
   We deduce from the approximation theorem that $b_1^{(2)}(F_d) = d-1$.
   
   \medskip
   
   (b) Let $G$ be a non-trivial group with a finite classifying space $BG$.
       Assume that $G$ admits a finite index normal tower $(G_i)_{i\in\bbN}$ such that $G_i \cong G$ for all $i$.
       Such a tower exists, for example, for a free abelian group $G = \bbZ^n$ and for the $3$-dimensional Heisenberg group $G= H_3(\bbZ)$.
       L\"uck's approximation theorem implies $b_q^{(2)}(G) = 0$ for all $q$. Indeed, the term in the numerator is constant, whereas the index tends to infinity.
\end{example}

Now we begin with the proof of L\"uck's approximation theorem; we proceed along the lines of an argument which Roman Sauer explained to me a couple of years ago.
As a first step, we reduce the statement of L\"uck's approximation theorem to an assertion on matrices over the group ring $\bbZ[G]$.
Recall that there is an involution $*$ of the complex group algebra $\bbC[G]$ defined for an element $a = \sum a_g g \in \bbC[G]$ 
by $a^* = \sum_{g} \bar{a}_g g^{-1}$. This involution has the property that $\rho(a^*)$ is the adjoint of $\rho(a)$ for
every unitary representation $\rho\colon G \to \U(\calH)$ of $G$.
Moreover, for all positive integers $m,n$, this extends to an operation $^*\colon \Mat_{m,n}(\bbC[G]) \to \Mat_{n,m}(\bbC[G])$
by imposing $(A^*)_{i,j} = A_{j,i}^*$.

Consider the cellular chain complex $C_\bullet(X)$ of $X$
\begin{equation*}
   C(X):\quad \dots C_{q+1}(X) \stackrel{\partial_{q+1}}{\longrightarrow} C_{q}(X) \stackrel{\partial_{q}}{\longrightarrow}  C_{q-1}(X) \longrightarrow \dots
\end{equation*}
The action of $G$ on $X$ is, by assumption, free and 
 $G\backslash X$ is of finite type, i.e.,
there are only finitely many orbits of cells in each dimension. We infer that $C_q(X)$ is a finitely generated free $\bbZ[G]$-module.
We fix a $\bbZ[G]$-basis of $C_q(X)$ for every $q\in \bbN_0$ by choosing a representative of every orbit of $q$-dimensional cells.
We get an isomorphism $C_q(X) \cong \bbZ[G]^{n_q}$ for some integer $n_q$.
Therefore we can describe the chain complex $C_\bullet(X)$ as
\begin{equation*}
 C(X):\quad \dots \bbZ[G]^{n_{q+1}} \stackrel{A_{q+1}}{\longrightarrow} \bbZ[G]^{n_{q}} \stackrel{A_{q}}{\longrightarrow}  \bbZ[G]^{n_{q-1}} \longrightarrow \dots
\end{equation*}
where the boundary operator $\partial_q$ can now be represented as the multiplication (from the right)
with a matrix $A_q \in \Mat_{n_q,n_{q-1}}(\bbZ[G])$.

For every finite index normal subgroup $N \trianglelefteq_{f.i.} G$, the cellular chain complex $C_\bullet(N \backslash X)$ is isomorphic to $\bbZ \otimes_{\bbZ[N]} C_\bullet(X)$.
In particular, we obtain
\begin{equation*}
    C(N\backslash X; \bbC)\colon\quad \dots \bbC[G/N]^{n_{q+1}} \stackrel{r_{N}(A_{q+1})}{\longrightarrow} \bbC[G/N]^{n_{q}} \stackrel{r_{N}(A_{q})}{\longrightarrow}  \bbC[G/N]^{n_{q-1}} \longrightarrow \dots
\end{equation*}
where $r_N(A_q)$ denotes the right multiplication operator defined by $A_q$ reduced modulo $N$.
By the Hodge-de Rham argument, the $q$-th complex homology $H_q(X/N,\bbC)$ space is isomorphic to
the kernel of the Laplacian $r_N(A_qA^*_q + A^*_{q+1}A_{q+1})$.

The cellular $L^2$-chain complex of $X$ w.r.t.~the action of $G$ is $C^{(2)}(X) := \ell^2(G) \otimes_{\bbZ[G]} C_\bullet(X)$. As before
the boundary operators are given by the operators 
\begin{equation*}
 r^{(2)}(A_q)\colon \ell^2(G)^{n_q} \to \ell^{2}(G)^{n_{q-1}}
\end{equation*}
defined by right multiplication with the matrices $A_q$.
For a given degree $q$ the $q$-th $L^2$-homology $H^{(2)}_q(X,\calN(G))$ 
is isomorphic to the kernel of the corresponding Laplacian 
$r^{(2)}( A_qA^*_q + A^*_{q+1}A_{q+1} )$ by the Hodge-de Rham argument; see \cite[Lemma 1.18]{LuckBook}.
We conculde that L\"uck's approximation theorem is a consequence the following algebraic approximation theorem for matrices.

\begin{theorem}[L\"uck Approximation for matrices]\label{thm:approximation-for-matrices}
 Let $G$ be a group and
 let $A \in \Mat_{n}(\bbZ[G])$ be a square matrix over the integral group ring.
 For every finite index normal tower $(G_i)_{i \in \bbN}$ in $G$ the following holds:
 \begin{equation*}
   \lim_{i \to \infty} \frac{\dim_{\bbC} \ker(r_{G_i}(A))}{[G:G_i]} = \dim_{\calN(G)} \ker(r^{(2)}(A)).
 \end{equation*}
\end{theorem}

\begin{remark}
 In the proof that we give here, it is essential that the matrix $A$ has entries in the integral group ring $\bbZ[G]$.
 It is easy to extend it to $\bbQ[G]$ and a similar proof allows to extend it $\overline{\bbQ}[G]$.
 Recently, Andrei Jaikin-Zapirain extended the result to matrices over $\bbC[G]$, which is a deeper result; see \cite{Jaikin2017}.
\end{remark}

\begin{definition}
  Let $A \in \Mat_n(\bbC[G])$. 
  We say that $A$ is \emph{globally positive}, if $A^* = A$ and $\rho(A) \colon \calH^n \to \calH^n$ is a positive
  operator for every unitary representation $\rho\colon G \to \U(\calH)$ of $G$.
\end{definition}

\begin{example} 
  Let $A \in \Mat_{m,n}(\bbC[G])$, then $A^*A$ and $AA^*$ are globally positive.
\end{example}

The matrix $A_qA^*_q + A^*_{q+1}A_{q+1}$ which defines the Laplacian is globally positive. Therefore we will restrict our discussion to this case and
prove the L\"uck approximation for globally positive matrices $A \in \Mat_n(\bbZ[G])$ over the integral group ring.
It is an easy exercise to deduce the general approximation for matrices from the specific result on globally positive matrices.

\begin{exercise}
 Assume that the approximation property in Theorem \ref{thm:approximation-for-matrices} 
 is known for all globally positive matrices $A \in \Mat_n(\bbZ[G])$ over the integral group ring.
 Deduce that it holds for arbitrary matrices $A \in \Mat_n(\bbQ[G])$ over the rational group ring. 
\end{exercise}

\begin{lemma}[Universal bound for the operator norm]\label{lem:universal-bound-operator-norm}
 For every $A \in \Mat_n(\bbC[G])$ there is a positive real number $\kappa_A > 0$, 
 such that for every unitary representation $\rho\colon G \to \U(\calH)$ of $G$
 the operator norm of $\rho(A)\colon \calH^n \to \calH^n$ is bounded above by $\kappa_A$, i.e., 
 \begin{equation*}
    \Vert\rho(A)\Vert_{\calH^n} \leq \kappa_A.
 \end{equation*}
\end{lemma}
\begin{proof}
   Let  $B \in \Mat_n(\bbC)$ be a matrix with only one non-zero entry $g \in G$.
   It is easy to see $\rho(B)$ is a local isometry in every unitary representation $\rho\colon G \to \U(\calH)$, this means $\rho(B)\rho(B^*)$ is a projection, 
   and thus $\Vert\rho(B)\Vert \leq 1$.
   A general matrix $A$ is a finite linear combination of
   such elementary matrices $B$. Therefore we can define $\kappa_A$ to be the sum of the absolute values
   of all the coefficients in all entries of $A$.
\end{proof}

\subsection{Spectral measures and weak convergence}

Let $\calM(\bbR)$ denote the set of finite, compactly supported Borel measures on $\bbR$.

We recall the notion of spectral measures.
Let $\calH$ be a finite dimensional complex Hilbert space.
Let $S\colon \calH \to \calH$ be a linear operator which is self-adjoint (i.e., $S^* = S$) and positive (i.e., $\langle Sx,x\rangle \geq 0$ for all $x\in \calH$).
We know that such an operator is diagonalisable and has only non-negative real eigenvalues. In addition, the set of eigenvalues is bounded above by the operator norm $\Vert S \Vert$.
For every Borel subset $A \subseteq \bbR$ we define 
$\mu_S(A)$ to be the number of eigenvalues of $S$ (counted with multiplicities) which lie in $A$. 
This defines a measure  $\mu_S \in \calM(\bbR)$ which is called the \emph{spectral measure} of $S$.
Observe that the equality
\begin{equation*}
     \Tr(P(S)) = \int_\bbR P(t) d\mu_S(t)
\end{equation*}
holds for every polynomial $P \in \bbC[T]$. Moreover, by definition $\dim_\bbC \ker(S) = \mu_S(\{0\})$.
In the same sense a spectral measure exists for positive operators on finitely generated Hilbert $\calN(G)$-modules.

\begin{lemma}
 For every positive operator $S\colon \ell^2(G)^n \to \ell^2(G)^n$ of Hilbert $\calN(G)$-modules, there is a unique
 measure $\mu_S\in \calM(\bbR)$, with support in $[0,\Vert S \Vert]$, such that
 \begin{equation*}
     \Tr_{\calN(G)}(P(S)) = \int_\bbR P(t) d\mu_S(t) 
\end{equation*}
for every polynomial $P \in \bbC[T]$. In addition $\dim_{\calN(G)} \ker(S) = \mu_S(\{0\})$.
\end{lemma}

\begin{definition}
  A sequence $(\mu_n)_{n\in \bbN}$ of measures $\mu_n \in \calM(\bbR)$ \emph{converges weakly} to $\mu \in \calM(\bbR)$, and we write 
  $\mu_n \stackrel{w}{\longrightarrow} \mu$,
  if 
  \begin{equation*}
      \lim_{n\to\infty} \int_{\bbR} f(t) d\mu_n(t) = \int_{\bbR} f(t) d\mu(t)
  \end{equation*}
  holds for every bounded continuous function $f\colon \bbR \to \bbC$. 
\end{definition}

\begin{exercise}\label{ex:weak-convergence}
 (a) Let $\mu \in \calM(\bbR)$ and let $(\mu_n)_{n\in \bbN}$ be a sequence of measures. Assume that 
 there is a compact set $C$ which contains the support of $\mu_n$ for every $n\in \bbN$ and assume further that
 \begin{equation*}
      \lim_{n\to\infty} \int_{\bbR} P(t) d\mu_n(t) = \int_{\bbR} P(t) d\mu(t)
  \end{equation*}
  holds for every polynomial $P \in \bbC[T]$. Show that $\mu_n \stackrel{w}{\longrightarrow} \mu$.\\
  \emph{Hint:} use the approximation theorem of Weierstra{\ss}.
  
  (b) Let $\mu_n \in \calM(\bbR)$ be a sequence of measures which converges weakly to $\mu \in \calM(\bbR)$. 
      Prove the \emph{Portmanteau Theorem}:
      the inequality
         \begin{equation*}
          \limsup_{n\to\infty} \int_{\bbR} f(t) d\mu_n(t)  \leq  \int_{\bbR} f(t) d\mu(t).
         \end{equation*}
      holds for every upper semi-continuous function $f\colon \bbR \to \bbR$ which is bounded from above.
      Deduce that $\limsup_{n\to\infty} \mu_n(A) \leq \mu(A)$ for every closed subset $A \subseteq \bbR$
      and $\liminf_{n\to\infty} \mu_n(U) \geq \mu(U)$ for every open subset $U \subset \bbR$.\\
      (Hint: every upper semi-continuous function is the pointwise limit of a non-increasing sequence of bounded continuous functions; this is known as Baire's theorem.)
\end{exercise}

\begin{lemma}[Weak convergence of spectral measures]\label{lem:weak-convergence}
  Let $G$ be a group and suppose $(G_i)_{i\in \bbN}$ is a finite index normal tower in $G$.
  For every globally positive matrix $A \in \Mat_n(\bbC[G])$,
  the sequence of normalized spectral measures $\mu_i := \frac{1}{[G:G_i]} \mu_{r_{G_i}(A)}$ 
  converges weakly to the $L^2$-spectral measure $\mu_{r^{(2)}(A)}$ as $i\to\infty$.
\end{lemma}
\begin{proof}
   By Lemma \ref{lem:universal-bound-operator-norm} all the measures are supported in a compact interval $[0,\kappa_A]$.
   It suffices (see Exercise \ref{ex:weak-convergence}) to verify 
   \begin{equation*}
      \lim_{i\to\infty} \int_{\bbR} P(t) d\mu_i(t) = \int_{\bbR} P(t) d\mu_{r^{(2)}(A)}(t)
  \end{equation*}
  for every polynomial $P \in \bbC[T]$. However, since $A^k$ is again globally positive for every $k\in\bbN$, it actually suffices to check 
  the above for the polynomial $P(T) = T$. In other words, we need to verify
  \begin{equation*}
      \lim_{i\to\infty} \frac{1}{[G:G_i]}\Tr(r_{G_i}(A))  = \Tr_{\calN(G)}(r^{(2)}(A)).
  \end{equation*}
  First assume that $n = 1$, i.e., that $A = \sum_{g\in G} c_g g \in \bbZ[G]$ is an element in the group ring.
  In this case $\Tr(r_{G_i}(A)) = [G:G_i] \sum_{g \in G_i} c_g$ since every $g \in G\setminus G_i$ acts like a fixed point free permutation on the standard basis vectors.
  On the other hand,
  $\Tr_{\calN(G)}(r^{(2)}(A)) = c_1$. The support of $A$ is finite and $\bigcap_{i\in\bbN} G_i = \{1\}$,  consequently
  $\frac{1}{[G:G_i]}\Tr(r_{G_i}(A)) = \Tr_{\calN(G)}(r^{(2)}(A))$ for all sufficiently large $i \in \bbN$.
  
  The assertion for general matrices $A \in \Mat_n(\bbC[G])$ follows from the same argument applied to every diagonal entry of the matrix $A$.
\end{proof}
\begin{example}
  In general, weak convergence is not sufficient to prove approximation.
  Let $\mu_i$ be the point measure at $\frac{1}{i}$. 
  This sequence of measures converges weakly to the point measure at $0$, however
  $ 0 = \lim_{i\to\infty} \mu_i(\{0\}) \neq  \mu(\{0\}) = 1$.
\end{example}

Nevertheless, the weak convergence of spectral measures obtained in Lemma \ref{lem:weak-convergence} and the 
Portmanteau Theorem (Exercise \ref{ex:weak-convergence})
immediately yield Kazhdan's inequality
\begin{equation*}
   \limsup_{i \to \infty} \frac{\dim_{\bbC} \ker(r_{G_i}(A))}{[G:G_i]} \leq \dim_{\calN(G)} \ker(r^{(2)}(A)).
 \end{equation*}
for every globally positive matrix $A \in \Mat_n(\bbC[G])$. 

\subsection{Fuglede-Kadison determinants and the proof of the Approximation Theorem}

\begin{definition}
 Let $\mu \in \calM(\bbR)$ be a measure. The real number
 \begin{equation*}
    \det(\mu) = \exp \int_{0^+}^\infty \ln(t) d\mu(t)
 \end{equation*}
  is called the \emph{Fuglede-Kadison determinant} of $\mu$.
  In particular, the Fuglede-Kadison determinant of a positive operator $S \colon \ell^{2}(G)^n \to \ell^2(G)^n$ of Hilbert $\calN(G)$-modules, is
  defined to be  $\det_{\calN(G)}(S) = \det(\mu_S)$ the determinant of its spectral measure. 
\end{definition}

\begin{lemma}\label{lem:bounded-FK-det}
 Let $\mu_i, \mu \in \calM(\bbR)$ measures supported on a compact interval $[0,C]$.
 Suppose that $\mu_i \stackrel{w}{\longrightarrow} \mu$ and that $\det(\mu_i) \geq \delta > 0$, 
 then $\lim_{i \to \infty}\mu_i(\{0\}) = \mu(\{0\})$.
\end{lemma}
\begin{proof}
  We know from the Portmanteau Theorem that $\limsup_{i\to\infty} \mu_i(\{0\}) \leq \mu(\{0\})$.
  
  Let $1 > \varepsilon > 0$ be given. We observe that
  \begin{equation*}
     \mu_i\bigl((0,\varepsilon)\bigr) \leq \frac{1}{|\ln(\varepsilon)|} \int_{0^+}^\varepsilon |\ln(t)| d\mu_i(t)
     \leq \frac{-\ln(\det(\mu_i)) + \ln(C) \mu_i(\bbR)}{|\ln(\varepsilon)|} \leq \frac{L}{|\ln(\varepsilon)|}
  \end{equation*}
  for all sufficiently large $i$ and some constant $L = L(C,\delta, \mu(\bbR)) > 0$.
  
  A short calculation yields
 \begin{align*}
    \liminf_{i\to\infty} \mu_i(\{0\}) + \frac{L}{|\log(\varepsilon)|}
    &\geq \liminf_{i\to\infty} \mu_i(\{0\}) + \limsup_{i\to \infty} \mu_i\bigl((0,\varepsilon)\bigr) \\
    &\geq \liminf_{i\to\infty} \mu_i\bigl((-\varepsilon,\varepsilon)\bigr) \geq \mu\bigl((-\varepsilon,\varepsilon)\bigr) \\
    &\geq \mu(\{0\}) 
    \end{align*}
where we use the Portmanteau Theorem (Exercise \ref{ex:weak-convergence}) in the third step.
Since $\varepsilon > 0$ was arbitrary, the theorem follows by taking the limit as $\varepsilon \to \infty$.
\end{proof}

\begin{proof}[Proof of Theorem \ref{thm:approximation-for-matrices}]
 Let $G$ be a group and let $(G_i)_{i\in \bbN}$ be a finite index normal tower.
 We consider a globally positive matrix $A \in \Mat_n(\bbZ[G])$ and the associated sequence
 $\mu_i := \frac{1}{[G:G_i]} \mu_{r_{G_i}(A)}$ of normalized spectral measures.
 We want to show that $\lim_{i\to\infty} \mu_i(\{0\}) = \mu_{r^{(2)}(A)}(\{0\})$ using Lemma \ref{lem:bounded-FK-det}.
 To this end we prove the L\"uck Lemma: $\det(\mu_i) \geq 1$.
 A power of the Fuglede-Kadison determinant $\det(\mu_i)^{[G:G_i]}$ is the product over all non-zero eigenvalues of $r_{G_i}(A)$.
 Since $r_{G_i}(A)$ can be represented by a matrix with integer entries, this 
 is a positive \emph{integer}. Indeed, it is the lowest non-zero coefficient of the characteristic polynomial of $r_{G_i}(A)$.
\end{proof}

\newpage 

\subsection{Generalizations and the Approximation Conjecture}
\subsubsection{Why normal subgroups?}
The notion of \emph{finite index normal towers} in the approximation theorem can be significantly weakened.
For every finite index subgroup $H \leq G$, we write $c_G(H)$ to denote the number subgroups conjugate to $H$. 
Similarly, for every element $g \in G$ we write $c_G(g,H)$ for the number 
of subgroups conjugate to $H$ which contain $g$.
\begin{definition}
  Let $G$ be a group. 
  A \emph{Farber sequence} in $G$ is a sequence $(G_i)_{i\in\bbN}$ of finite index subgroups of $G$
  such that 
  \begin{equation*}
     \lim_{i \to \infty}  \frac{c_G(g,G_i)}{c_G(G_i)} = 0
  \end{equation*}
  holds for every element $g \in G$ with $g\neq 1$.
\end{definition}

\begin{exercise}
 (a) Show that a finite index normal tower $(G_i)_{i\in\bbN}$ is a Farber sequence.
 
 \smallskip
 
 (b) Give an example of a group $G$ and a Farber sequence $(G_i)_{i\in \bbN}$ such that none of the subgroups $G_i$ is normal in $G$.
 
  \smallskip
 
 (c) Prove that every group which admits a Farber sequence is residually finite.
\end{exercise}

\begin{theorem}[Farber \cite{Farber1998}]
  Let $X$ be a free $G$-CW-complex such that $G \backslash X$ is of finite type. Every Farber sequence $(G_i)_{i\in\bbN}$ in $G$ satisfies
  \begin{equation*}
     \lim_{i\to\infty} \frac{b_q(G_i \backslash X; \bbQ)}{[G:G_i]} = b_q^{(2)}(X;\calN(G)).
  \end{equation*}
\end{theorem}
\begin{exercise}
  Let $G$ be a group with a Farber sequence $(G_i)_{i\in \bbN}$.
  
 (a) Consider an element $a = \sum_{g\in G} a_g g \in \bbC[G]$ of the complex group algebra.
    Show that $\lim_{i\to\infty} [G:G_i]^{-1} \Tr(r_{G_i}(a)) = a_1$. Use this to prove the weak convergence of spectral measures (Lemma \ref{lem:weak-convergence}) for Farber sequences.
    
    \smallskip
    
 (b) Prove the approximation theorem of Farber by adapting the proof of L\"uck's theorem.
\end{exercise}

\subsubsection{Why finite index subgroups?}

\begin{conjecture}[Approximation Conjecture, cf.\ Conjecture 13.1 in \cite{LuckBook}]
Let $G$ be a group and let $X$ be a free $G$-CW-complex such that $G \backslash X$ is of finite type.
For every normal tower $(G_i)_{i\in\bbN}$ in $G$
\begin{equation}\label{eq:approximation-general}
   \lim_{i\to \infty} b_q^{(2)}(G_i \backslash X,\calN(G/G_i)) = b_q^{(2)}(X,\calN(G))
\end{equation}
for every $q\in \bbN_0$.
\end{conjecture}

As before, this conjecture can be reformulated in terms of matrices. 
This conjecture generalizes L\"uck's approximation theorem, due to the following observation:
if $G/G_i$ is finite, then $b_q^{(2)}(G_i \backslash X,\calN(G/G_i)) = \frac{b_q(G_i \backslash X,\bbQ)}{[G:G_i]}$.

In view of Lemma \ref{lem:bounded-FK-det}, the Approximation Conjecture follows from the so-called Determinant Conjecture.
The relevance of the Determinant Conjecture for approximation was first realized by Schick \cite{Schick01}.
\begin{conjecture}[Determinant Conjecture, cf.\ 13.2 in \cite{LuckBook}]
Let $G$ be a group and let $A \in \Mat_n(\bbZ[G])$ be a globally positive matrix.
Then the Fuglede-Kadison determinant satisfies $\det_{\calN(G)}(r^{(2)}(A)) \geq 1$.
\end{conjecture}

In general these two conjectures are open and it is not clear if one should really expect them to hold for all groups. Nevertheless, it is of interest to understand precisely when 
they are satisfied. 
If $G$ is a group and $(G_i)_{i\in\bbN}$ is a normal tower such that \eqref{eq:approximation-general} holds for every $X$,
then we say that $G$ satisfies the Approximation Conjecture w.r.t.\ $(G_i)_i$.
\begin{theorem}[Elek-Szab\'o \cite{ElekSzabo2005}]
  (a) Every sofic group satisfies the Determinant Conjecture.
  
  \smallskip
  
  (b) Let $G$ be a group and let $(G_i)_{i\in\bbN}$ be a normal tower.
      Assume that every quotient $G/G_i$ is \emph{sofic}, then $G$ satisfies the Approximation Conjecture w.r.t.\ $(G_i)_{i\in\bbN}$. 
\end{theorem}

The \emph{sofic groups} were introduced by Gromov \cite{Gromov1999} -- under a different name --
and studied further by B.\ Weiss \cite{Weiss2000}. We briefly recall the definition given in \cite[Def.~1.2]{ElekSzabo2005} and we refer to \cite{ElekSzabo2006} for further properties.

\begin{definition}\label{def:soficgroup}
A group $G$ is \emph{sofic} if for every $\varepsilon >0$ and every finite subset $W \subseteq G$,
there exists a finite set $X$ and a function $f \colon G \to \Aut(X)$ (i.g.\ no homomorphism) such that
\begin{enumerate}
\item for all $u,v \in W$ the set 
\begin{equation*}
L_{u,v} = \{\:x\in X\:|\: f(uv)(x) = (f(u)\circ f(v))(x) \:\}
\end{equation*}
has at least $(1-\varepsilon)|X|$ elements.
\item $f(1) = \id_X$ and for all $u\in W\setminus \{1\}$ the fixed point set $\fix_X(f(u))$ has at most $\varepsilon|X|$ elements. 
\end{enumerate}
\end{definition}
Every residually finite group is sofic and every amenable group is sofic.
The class of sofic groups is closed under direct limits, inverse limits, taking subgroups and free products. Moreover, sofic-by-amenable extensions are sofic (see \cite[Thm.\ 1]{ElekSzabo2006}).
The class of sofic group is so large, that at present no example of a non-sofic group is known. 
This leads to an important open problem:
\begin{question}
  Is every group sofic?
\end{question}

\section{The growth of Betti numbers in positive characteristic}

Let $p$ be a prime number.
In the previous section we discussed the approximation theorem, which identifies the $L^2$-Betti number as the limit of a sequence of 
normalized $\bbQ$-Betti numbers. Is there an analogous theory for the $\bbF_p$-Betti numbers?

\begin{question}\label{qu:convergence-mod-p}
 Let $G$ be a residually finite group with a finite index normal tower $(G_i)_{i\in\bbN}$.
 Suppose that $X$ is a free $G$-CW-complex such that $G\backslash X$ is of finite type.
 Does the sequence of normalized $q$-th $\bbF_p$-Betti numbers
 \begin{equation*}
    \frac{b_q(G_i\backslash X; \bbF_p)}{[G:G_i]}
 \end{equation*}
 converge as $i$ tends to infinity?
 If yes, what is the limit? Does it depend on the finite index normal tower?
\end{question}

\begin{remark}
 By the universal coefficient theorem we always have $b_q(G_i\backslash X; \bbF_p) \geq b_q(G_i\backslash X; \bbQ)$.
 Hence, if the limit exists, it is greater or equal than the corresponding the $L^2$-Betti number.
\end{remark}

\begin{conjecture}[Conjecture 2.4 \cite{LuckSurvey2016}]\label{conj:mod-p-approx}
 Let $X$ be a \emph{contractible} free $G$-CW-complex s.t.\ $G\backslash X$ is of finite type. 
 For every finite index normal tower $(G_i)_{i\in \bbN}$, the limit
 \begin{equation*}
    \lim_{i\to \infty} \frac{b_q(G_i\backslash X, \bbF_p)}{[G:G_i]} = b^{(2)}_q(X,\calN{G})
 \end{equation*}
 exists and is the $L^2$-Betti number.
\end{conjecture}

\begin{example}[cf.\ Example 6.2 in \cite{LLS2011}]
 Here is an example to illustrate why $X$ is assumed to be contractible in Conjecture \ref{conj:mod-p-approx}. 
 Let $d \geq 2$ and take a map $f\colon S^d \to S^d$ of degree $p$, i.e., $H_d(f)$ is multiplication with $p$.
 The existence of such a map can be seen by induction on $d$.
 Construct $Y = S^d \cup_f D^{d+1}$ by attaching a $(d+1)$-cell using the map $f$.
 The space $Y$ is simply connected.
 Moreover, the integral homology of $Y$ vanishes in every degree $q \geq 1$, except for degree $q = d$ where $H_d(Y) \cong \bbZ/p\bbZ$ by construction.
 
 Now consider the CW-complex $X_0 = S^1 \vee Y$ with fundamental group $G = \pi_1(X_0) \cong \bbZ$.
 The universal covering $X$ of $X_0$ is the real line $\bbR$ with a copy of $Y$ attached at every integer point; it is a free $G$-CW-complex.
 
 For $i \in \bbN$ we denote by $G_i \leq G$ the unique subgroup of index $i$. The space $X_i =  G_i\backslash X$ is a circle with $i$ copies of $Y$ attached.
 In particular, we obtain $H_d(X_i,\bbQ) = 0$ but $H_d(X_i, \bbF_p) = \bbF_p^i$.
 Therefore we get
 \begin{equation*}
   \lim_{i\to \infty} \frac{b_q(G_i\backslash X, \bbF_p)}{[G:G_i]} = 1  \: > \: 0 = \frac{b_q(G_i\backslash X, \bbQ)}{[G:G_i]} = b^{(2)}_q(X,\calN{G}).
 \end{equation*}
\end{example}

\medskip

In general it is unknown whether the sequence of normalized mod-$p$ Betti numbers converges at all. However, the sequence converges if
$[G:G_i]$ is a $p$-power for every $i$. 
\begin{theorem}[Bergeron-Linnell-L\"uck-Sauer, Thm.~1.6 \cite{BLLS2014}]\label{thm:mod-p-convergence}
 Let $G$ be a group and let $X$ be a free $G$-CW-complex such that $G\backslash X$ is of finite type.
 Assume that $(G_i)_{i\in\bbN}$ is a decreasing sequence of finite index subgroups $G_1 \geq G_2 \geq G_3 \dots$ in $G$
 such that $G_{i+1}$ is normal in $G_i$ and $[G_i:G_{i+1}]$ is a $p$-power for every $i$.
 For every $q \in \bbN_0$ the sequence of $q$-th $\bbF_p$-Betti numbers
  \begin{equation*}
    \frac{b_q(G_i\backslash X; \bbF_p)}{[G:G_i]}
 \end{equation*}
 is monotonically decreasing and converges as $i$ tends to infinity.
\end{theorem}

This result is a consequence of some elementary observations about modular representation theory.
 Let $R = \bbF_p[\bbZ/p\bbZ]$ be the group ring of the finite cyclic group of order $p$.
 The ring $R$ is isomorphic to $\bbF_p[T]/(T^p-1)$. However, $T^p-1 = (T-1)^p$ in a field of characteristic $p$; so we may write $\tau = T-1$ to obtain an isomorphism
 $R \cong \bbF_p[\tau]/(\tau^p)$. The ring $R$ is a local ring with unique maximal ideal $\tau R$ and every element $r \in R$
 can be written $r = \tau^k u$ for a unique power $k \in \{0,\dots,p\}$ and some unit $u \in R^\times$.
 
\begin{lemma}\label{lem:modular-lemma}
Let $\varphi\colon M \to M$ be a homomorphism of free finite rank $R$-modules
 and let $\overline{\varphi}\colon M/\tau M \to N/\tau N$ be the induced $\bbF_p$-linear map.
 Then the following two inequalities hold:
 \begin{align*}
     \dim_{\bbF_p} \im(\varphi) \geq p \dim_{\bbF_p} \im(\overline{\varphi})\\
     \dim_{\bbF_p} \ker(\varphi) \leq p \dim_{\bbF_p} \ker(\overline{\varphi})
 \end{align*}
\end{lemma}
\begin{proof}
 Due to the observation above, the Gau{\ss}-Elimination algorithm works in the ring $R$! Therefore, we can find $R$-bases of $M$ and $N$ such that
 the matrix $A$ of $\varphi$ with respect to the chosen bases is diagonal with entries of the form $\tau^k$ for exponents $k \in \{0,\dots,p\}$.
 Let $j$ be the number of entries of the form $1 = \tau^0$.
 Then  $\dim_{\bbF_p} \im(\overline{\varphi}) = j$, whereas
 $\dim_{\bbF_p} \im(\varphi)$ is at least $j\dim_{\bbF_p} R = j p$.
\end{proof}

\begin{proof}[Proof of Theorem \ref{thm:mod-p-convergence}]
 Every finite $p$-group is nilpotent, so we may refine the sequence $(G_i)_{i\in\bbN}$ so that $G_i/G_{i+1} \cong \bbZ/p\bbZ$.
 
 We have to show that $\frac{b_q(G_{i+1}\backslash X; \bbF_p)}{p} \leq b_q(G_i\backslash X; \bbF_p)$.
 The CW-complex $Y = G_{i+1}\backslash X$ carries a free action of $G_i/G_{i+1} \cong \bbZ/p\bbZ$. 
 We choose a set of representatives for the orbits of cells in every dimension to see that the
 mod-$p$ cellular chain complex of $Y$ is a chain complex of finite rank free modules over $R = \bbF_p[\bbZ/p\bbZ]$, i.e.,
 \begin{equation*}
 C(Y; \bbF_p):\quad \dots R^{n_{q+1}} \stackrel{\partial_{q+1}}{\longrightarrow} R^{n_{q}} \stackrel{\partial_{q}}{\longrightarrow}  R^{n_{q-1}} \longrightarrow \dots
\end{equation*}
Now Lemma \ref{lem:modular-lemma} implies
\begin{align*}
b_q(G_{i+1}\backslash X; \bbF_p) &= \dim_{\bbF_p} \ker(\partial_q) - \dim_{\bbF_p}\im(\partial_{q+1})\\ 
&\leq p \dim_{\bbF_p} \ker(\overline{\partial_q}) - p \dim_{\bbF_p}\im(\overline{\partial_{q+1}}) = p \: b_q(G_{i}\backslash X; \bbF_p) \qedhere
\end{align*}
\end{proof}

\subsection{The growth of Betti numbers in $p$-adic analytic towers}

The main step towards more precise results is to find a good candidate for the limit of the sequence of Betti numbers.
In the case of $L^2$-Betti numbers, we embedded the group ring $\bbZ[G]$ into the von Neumann algebra $\calN(G)$ and
used the dimension theory of von Neumann algebras to define $L^2$-Betti numbers.
In this vein, one hopes that it is possible to embed the group algebra $\bbF_p[G]$
into a ring with a useful notion of dimension of modules.
In some cases one can embed the group algebra $\bbF_p[G]$ even into a division algebra; this applies, for instance, if $G$ is a
torsion-free elementary amenable group \cite{LLS2011} or if $G$ is orderable; see \cite{BLLS2014} for a related result. However, in general this is not possible. In other cases, it is fruitful to restrict attention to a specific class of finite index normal towers.
Here we discuss the growth of Betti numbers in $p$-adic analytic towers, where both methods
can be combined to obtain strong results. This method was introduced by Calegari
and Emerton in \cite{CalegariEmerton09}; here we follow the more general account of Bergeron-Linnell-L\"uck-Sauer \cite{BLLS2014}.

\subsubsection{Compact $p$-adic Lie groups and Iwasawa algebras}
Let $p$ be a prime number. The ring of $p$-adic integers $\bbZ_p = \varprojlim_{k\in\bbN} \bbZ/p^k\bbZ$ is a complete discrete valuation ring
with maximal ideal $p\bbZ_p$. The field of fractions of $\bbZ_p$ is the field  $\bbQ_p$ of $p$-adic numbers, which is isomorphic to the completion of the field $\bbQ$ of rational numbers with respect to the $p$-adic valuation.
The group $\GL_N(\bbZ_p)$ of invertible $N\times N$-matrices over $\bbZ_p$ is a compact, totally disconnected topological group.
A basis of open neighbourhoods of the identity is given by the principal congruence subgroups, i.e. the kernels of the surjective homomorphisms $\red_{p^k} \colon \GL_N(\bbZ_p) \to \GL_N(\bbZ/p^k\bbZ)$.

\begin{definition}
 A topological group $H$ is a \emph{compact $p$-adic Lie group}  if it is isomorphic to a closed subgroup of $\GL_N(\bbZ_p)$ for some $N$.
\end{definition}
Sometimes compact $p$-adic Lie groups are also called compact $p$-adic analytic groups.
Just as real Lie groups these groups can also be defined as group objects in the category of $p$-adic analytic manifolds.
This will not be important for us, except that the \emph{dimension} of a $p$-adic Lie group is defined as the dimension of the underlying manifold.

Every compact $p$-adic Lie group contains a neighbourhood basis of open compact subgroups which are \emph{uniform pro-$p$ groups} (also called \emph{uniformly powerful} pro-$p$ groups).
These groups are the main tool in the theory of $p$-adic Lie groups since they admit a useful Lie theory. For a concise definition we refer to
\cite[Chapter 4]{DDMS}.
Here it suffices to mention that the principle congruence subgroups
$\ker(\red_{p^i}) \leq \GL_N(\bbZ_p)$ are uniform for all $i\geq 1$ if $p$ is odd and for all $i\geq 2$ if $p=2$.

Let $H \leq \GL_N(\bbZ_p)$ be a compact $p$-adic Lie group. The \emph{Iwasawa algebra} of $H$ is the profinite ring
\begin{equation*}
   \Iwa{H} = \varprojlim_{U \trianglelefteq_o H} \bbF_p[H/U]
\end{equation*}
where the limit is taken over the inverse system of open normal subgroups of~$H$.
The Iwasawa algebra is particularly well understood if $H$ is a uniform pro-$p$ group.
In this case the Iwasawa algebra $\Iwa{H}$ is a left and right noetherian ring which has no zero-divisors; see 7.25 in \cite{DDMS}.
Moreover, in this case the non-zero elements of $\Iwa{H}$ satisfy the left and right Ore condition. 
The Ore localization $D_H$ of $\Iwa{H}$ is a division algebra; it is the \emph{classical ring of fractions} of $\Iwa{H}$.

Now, let $H$ be any compact $p$-adic Lie group. As mentioned above, $H$ contains an open normal uniform pro-$p$ subgroup $U\trianglelefteq_o H$.
The Iwasawa algebra $\Iwa{H}$ is a finitely generated free module over $\Iwa{U}$ of rank $[H:U]$.
We will use this to define the \emph{rank} of finitely generated $\Iwa{H}$ modules.
\begin{definition}
   Let $M$ be a finitely generated $\Iwa{H}$-module and let $U\trianglelefteq_o H$ be an open normal uniform pro-$p$ subgroup.
   The \emph{rank} $\rnk_{\Iwa{H}}(M)$ of $M$ is 
   \begin{equation*}
     \rnk_{\Iwa{H}}(M) = \frac{1}{[H:U]} \dim_{D_U} \bigl( D_U \otimes_{\Iwa{U}} M \bigr).
   \end{equation*}
\end{definition}
The rank is independent of the chosen open normal uniform pro-$p$ subgroup $U$.
Moreover, the rank is additive in short exact sequences.
\begin{lemma}\label{lem:rank-is-additive}
  Let $H$ be a compact $p$-adic Lie group.
  Then $\rnk_{\Iwa{H}}(\Iwa{H}) = 1$ and
  every short exact sequence 
  \begin{equation*}
     0 \longrightarrow L_1 \longrightarrow L_2 \longrightarrow L_3 \longrightarrow 0
  \end{equation*}
   of finitely generated $\Iwa{H}$-modules, the ranks satisfy
   \begin{equation*}
  \rnk_{\Iwa{H}}(L_2) = \rnk_{\Iwa{H}}(L_1) + \rnk_{\Iwa{H}}(L_3).
  \end{equation*}
\end{lemma}
\begin{exercise}
 Give a proof of Lemma \ref{lem:rank-is-additive}.\\
 (Hint: use the fact that the Ore localization $D_U$ of the Iwasawa algebra $\Iwa{U}$ is a flat $\Iwa{U}$-module, when $U\trianglelefteq_o H$ is an open normal uniform pro-$p$ subgroup.)
\end{exercise}

\subsubsection{The theory of $p$-adic analytic towers}
\begin{definition}
Let $G$ be a group and let $(G_i)_{i \in \bbN}$ be a finite index normal tower.
The tower is said to be \emph{$p$-adic analytic}, if there is an injective homomorphism $\varphi\colon G \to \GL_N(\bbZ_p)$,
such that $G_i = \ker(\red_{p^i} \circ \varphi)$. The closure $H$ of the image $\varphi(G)$ of $\varphi$ is a compact $p$-adic Lie group; its dimension $d$ is called the \emph{dimension} of the tower.
\end{definition}

\begin{remark}
Let $(G_i)_i$ be a $p$-adic analytic tower in $G$. The dimension $d$ and the associated compact $p$-adic Lie group $H$ are independent of the 
homomorphism $\varphi\colon G \to \GL_N(\bbZ_p)$.
\end{remark}

 Let $G$ be a group and let $(G_i)_{i\in\bbN}$ be a $p$-adic analytic tower with associated $p$-adic Lie group $H$ of dimension $d$.
 Let $X$ be a free $G$-CW-complex such that $G\backslash X$ is of finite type.
 
 \begin{definition}
   The \emph{$q$-th mod-$p$ $L^2$-Betti number} of $X$ with respect to the tower $(G_i)_{i\in\bbN}$
   is
   \begin{equation*}
     \beta_q(X,G;\bbF_p) = \rnk_{\Iwa{H}}\Bigl(H_q\bigl(\Iwa{H}\otimes_{\bbF_p[G]} C_\bullet(X,\bbF_p)\bigr)\Bigr).
   \end{equation*}
 \end{definition}

\begin{theorem}[Bergeron-Linnell-L\"uck-Sauer \cite{BLLS2014}]\label{thm:mod-p-approximation}
  Let $G$ be a group and let $(G_i)_{i\in\bbN}$ be a $p$-adic analytic tower of dimension $d$.
  Let $X$ be a free $G$-CW-complex such that $G\backslash X$ is of finite type, then for every $q\in \bbN_0$
  \begin{equation*}
       b_q(G_i\backslash X, \bbF_p) = \beta_q(X,G,\bbF_p) \: [G:G_i] + O\bigl([G:G_i]^{1-\frac{1}{d}}\bigr).
  \end{equation*}
\end{theorem}
\begin{remark}
 In this setting, there is rational number $w$ such that the index $[G:G_i]$ is equal to
 \begin{equation*}
   [G:G_i] = w p^{di}
 \end{equation*}
  for all sufficiently large $i\in\bbN$. In particular, $O\bigl([G:G_i]^{1-\frac{1}{d}}\bigr) = O(p^{i(d-1)})$ as $i$ tends to infinity.
\end{remark}

The key ingredient is the following result; a proof can be found in \cite{BLLS2014}.
\begin{theorem}[Harris]
Let $H \leq_c \GL_N(\bbZ_p)$ be a compact $p$-adic Lie group of dimension~$d$ and let $H_i = H \cap \ker(\red_{p^i})$.
For every finitely generated $\Iwa{H}$-module~$M$, the identity
\begin{equation*}
   \dim_{\bbF_p} (\bbF_p \otimes_{\Iwa{H_i}} M) = [H:H_i] \: \rnk_{\Iwa{H}}(M) + O\bigl([H:H_i]^{1-\frac{1}{d}}\bigr)
\end{equation*}
holds as $i \to \infty$.
\end{theorem}
\begin{corollary}\label{cor:dim-formulae}
 Let $f\colon L \to M$ be a homomorphism of free finite rank left $\Iwa{H}$-modules.
 For every $i\in\bbN$ we define $f^{[i]} := \id \otimes f \colon \bbF_p \otimes_{\Iwa{H_i}} L \to \bbF_p \otimes_{\Iwa{H_i}} M$.
 The following statements hold as $i\to\infty$:
 \begin{enumerate}
  \item\label{it:coker-formula} $\dim_{\bbF_p}\bigl(\coker(f^{[i]})\bigr) = \rnk_{\Iwa{H}}\bigl(\coker(f)\bigr) [H:H_i] + O\bigl([H:H_i]^{1-\frac{1}{d}}\bigr)$
  \item\label{it:im-formula} $\dim_{\bbF_p}\bigl(\im(f^{[i]})\bigr) = \rnk_{\Iwa{H}}\bigl(\im(f)\bigr) [H:H_i] + O\bigl([H:H_i]^{1-\frac{1}{d}}\bigr)$
  \item\label{it:ker-formula} $\dim_{\bbF_p}\bigl(\ker(f^{[i]})\bigr) = \rnk_{\Iwa{H}}\bigl(\ker(f)\bigr) [H:H_i] + O\bigl([H:H_i]^{1-\frac{1}{d}}\bigr)$
 \end{enumerate}
\end{corollary}
\begin{proof}
 The tensor product functors $\bbF_p \otimes_{\Iwa{H_i}} \cdot$ are right exact, this means, they preserve cokernels.
 Therefore $\coker(f^{[i]}) = \bbF_p \otimes_{\Iwa{H_i}} \coker(f)$ and assertion \eqref{it:coker-formula} follows from Harris' theorem.
 Now, we deduce \eqref{it:im-formula}.
 We have short exact sequences
 \begin{align*}
   0 \longrightarrow \:\im(f)\: \longrightarrow \:&M\: \longrightarrow \:\coker(f)\: \longrightarrow 0\\
   0 \longrightarrow \im(f^{[i]}) \longrightarrow \bbF_p &\otimes_{\Iwa{H_i}} M \longrightarrow \coker(f^{[i]}) \longrightarrow 0\\
 \end{align*}
 for all $i\in \bbN$. If $m \in \bbN$ is the rank of the free module $M$, then
 $m = \rnk_{\Iwa{H}}(M)$ and $m[H:H_i] = \dim_{\bbF_p} \bbF_p \otimes_{\Iwa{H_i}} M$.
 Finally, \eqref{it:im-formula} follows from the following calculation
 \begin{align*}
  \dim_{\bbF_p}\bigl(\im(f^{[i]})\bigr) &= m[H:H_i] - \dim_{\bbF_p}\bigl(\coker(f^{[i]})\bigr) \\
                                    &= \bigl(m- \rnk_{\Iwa{H}}\bigl(\coker(f)\bigr)\bigr)[H:H_i] + O\bigl([H:H_i]^{1-\frac{1}{d}}\bigr)\\
                                    &= \rnk_{\Iwa{H}}\bigl(\im(f)\bigr) [H:H_i] + O\bigl([H:H_i]^{1-\frac{1}{d}}\bigr).
 \end{align*}
 We leave it as an exercise to verify the last assertion.
\end{proof}
\begin{exercise}
  Prove statement \eqref{it:ker-formula} of Corollary \ref{cor:dim-formulae}.
\end{exercise}

\begin{proof}[Proof of Theorem \ref{thm:mod-p-approximation}]
We first note that $[G:G_i] = [H:H_i]$ since the image of $G$ is dense in $H$.
The cellular chain complex
$C_\bullet(X,\bbF_p)$ is a chain complex of finite rank free $\bbF_p[G]$-modules, thus
the chain complex $\widehat{C}_\bullet = \Iwa{H} \otimes_{\bbF_p[G]} C_\bullet$
\begin{equation*}
  \widehat{C}_\bullet\colon \quad 
 \dots \widehat{C}_{q+1} \stackrel{\widehat{\partial}_{q+1}}{\longrightarrow} \widehat{C}_{q} \stackrel{\widehat{\partial}_{q}}{\longrightarrow}  \widehat{C}_{q-1} \longrightarrow \dots
\end{equation*}
is a chain complex of finite rank free $\Iwa{H}$-modules.
By Lemma \ref{lem:rank-is-additive} the $q$-th mod-$p$ $L^2$-Betti number is 
$\beta_q(X,G;\bbF_p) = \rnk_{\Iwa{H}}(\ker(\widehat{\partial}_{q})) -  \rnk_{\Iwa{H}}(\im(\widehat{\partial}_{q+1}))$.
Moreover, we observe that the cellular chain complex $C_\bullet(G_i\backslash X, \bbF_p)$ of $G_i\backslash X$
is isomorphic to the tensor product $\bbF_p \otimes_{\Iwa{H_i}} \widehat{C}_\bullet$.
Finally, Corollary \ref{cor:dim-formulae} yields
\begin{align*}
  b_q(G_i\backslash X,\bbF_p) &= \dim_{\bbF_p} \bigl( \ker(\widehat{\partial}^{[i]}_q) \bigr) - \dim_{\bbF_p}(\widehat{\partial}^{[i]}_{q+1}) \\
                              &= [G:G_i] \: \bigl( \rnk_{\Iwa{H}}(\ker(\widehat{\partial}_{q})) - \rnk_{\Iwa{H}}(\im(\widehat{\partial}_{q+1})) \bigr) + O([G:G_i]^{1-\frac{1}{d}})\\
                              &= [G:G_i] \:\beta_q(X,G,\bbF_p) + O([G:G_i]^{1-\frac{1}{d}}).
\end{align*}
\end{proof}

\subsection{ The first Betti numbers of groups and the rank gradient }

Let $G$ be a finitely generated group. We take a closer look at the first $L^2$-Betti number of $G$.

The first homology group of $G$ is isomorphic to the abelianization:
\begin{equation*}
 H_1(G;\bbZ) \cong G^{ab} = G/[G,G].
\end{equation*}
The first Betti number of $G$ is the rank of the free part of this finitely generated abelian group.
The first mod-$p$ homology of $G$ is $H_1(G,\bbF_p) \cong G^{ab}/pG^{ab}$ and the first $\bbF_p$-Betti number is
the dimension of this $\bbF_p$-vector space. In particular, we obtain
\begin{equation}\label{eq:inequalities}
    b_1(G;\bbQ) \leq b_1(G,\bbF_p) \leq d(G),
\end{equation}
where $d(G)$ denotes the minimal number of generators of $G$.

 Given a subgroup of finite index $H \leq_{f.i.} G$, the Nielsen-Schreier formula for free groups implies that $d(H)-1 \leq [G:H](d(G)-1)$.
 In particular, for every finite index normal tower $(G_i)_{i \in \bbN}$, the sequence of non-negative numbers $\frac{d(G_i)-1}{[G:G_i]}$ is monotonically decreasing;
 therefore the limit 
 \begin{equation*}
    \RG(G,(G_i)) := \lim_{i\to\infty} \frac{d(G_i)-1}{[G:G_i]}
 \end{equation*}
 exists. The number $\RG(G,(G_i))$ is called the \emph{rank gradient} of $G$ w.r.t.\ $(G_i)_{i\in\bbN}$. The rank gradient was introduced by Lackenby in \cite{Lackenby05}.
 Even though it has been studied frequently in recent years, the following question remains open.
 
 \begin{question}\label{qu:rank-gradient-tower}
  Let $G$ be a finitely generated, residually finite group. Is $\RG(G,(G_i))$ independent of the chosen finite index normal tower?
 \end{question}

  Assume now that $G$ is an infinite, finitely presented, residually finite group.
  In this case L\"uck's approximation theorem and \eqref{eq:inequalities} imply the following inequalities
  \begin{equation*}
     b_1^{(2)}(G) \leq \liminf_{i \to \infty} \frac{b_1(G_i,\bbF_p)}{[G:G_i]} \leq \limsup_{i \to \infty} \frac{b_1(G_i,\bbF_p)}{[G:G_i]} \leq \RG(G,(G_i)).
  \end{equation*}
 It is thus tempting to ask the following question.
 \begin{question}[cf.\ Question 3.3 \cite{LuckSurvey2016}]\label{qu:L2betti-rank-gradient}
  Let $G$ be an infinite, residually finite, finitely presented group.
  Is it true that $\RG(G,(G_i)) = b^{(2)}_1(G)$ holds for every finite index normal tower $(G_i)_{i\in\bbN}$ in $G$?
 \end{question}
 In particular, an affirmative answer immediately provides answers to Question \ref{qu:rank-gradient-tower},
 Question \ref{qu:convergence-mod-p} and Conjecture \ref{conj:mod-p-approx} for $b_1$.
 Clearly, the answer to Question \ref{qu:L2betti-rank-gradient} is positive for all groups with vanishing rank gradient.
 For example, this holds for all finitely generated groups which contain an infinite amenable normal subgroup; see~\cite[Theorem 3]{AbertNikolov2012}.

\providecommand{\bysame}{\leavevmode\hbox to3em{\hrulefill}\thinspace}
\providecommand{\href}[2]{#2}

\end{document}